\documentclass[preprint]{elsarticle}

\usepackage{latexsym,amsmath,amssymb,color}
\usepackage{graphics, amsthm}

\newcommand{\diag}{{\rm diag}}

\newcommand{\bmat}{\left[ \begin{array}}
\newcommand{\emat}{\end{array} \right]}
\newcommand{\ignore}[1]{}
\newcommand{\tr}{}

\newtheorem{definition}{Definition}
\newtheorem{theorem}{Theorem}[section]
\newtheorem{corollary}{Corollary}[section]

\begin{document}
\journal{Journal of Computational and Applied Mathematics}

\begin{frontmatter}

\title{Bidiagonal Decompositions \\ of Vandermonde-Type Matrices of Arbitrary Rank\tnoteref{t1,t2}} 

\tnotetext[t1]{
This research was partially supported by Spanish Research Grant PGC2018-096321-B-I00(MCIU/AEI) and by Gobierno de Arag\'on
(E41-17R) . The authors A. Marco and J. J. Mart{\'\i}nez are members of the Reseach Group {\sc asynacs} (Ref. {\sc ct-ce2019/683)} of Universidad de Alcal\'a. 
} 

\tnotetext[t2]{This research was also partially
supported by the Woodward Fund for Applied Mathematics at San Jos\'e State University. The Woodward Fund is a gift from the estate of Mrs.~Marie Woodward in memory of her son, Henry Teynham Woodward. He was an alumnus of the Mathematics Department at San Jos\'e State University and worked with research groups at NASA Ames.}


\author[1]{Jorge Delgado} \ead{jorgedel@unizar.es}
\author[2]{Plamen Koev\corref{cor1}} \ead{plamen.koev@sjsu.edu}
\author[3]{Ana Marco} \ead{ana.marco@uah.es}
\author[4]{Jos\'e-Javier Mart\'inez} \ead{jjavier.martinez@uah.es}
\author[5]{Juan Manuel Pe\~na} \ead{jmpena@unizar.es}
\author[6]{Per-Olof Persson} \ead{persson@berkeley.edu}
\author[7]{Steven Spasov} \ead{stevenspasov@gmail.com}

\cortext[cor1]{Corresponding author}

\ignore{
\fntext[fn1]{This is the first author footnote.}
\fntext[fn2]{Another author footnote, this is a very long
footnote and it should be a really long footnote. But this footnote is not yet sufficiently long enough to make two lines of footnote text.}
\fntext[fn3]{Yet another author footnote.}
}

\affiliation[1]{organization={Departamento de Matem\' atica Aplicada, Universidad de Zaragoza}, addressline={Edificio Torres Quevedo}, city={Zaragoza}, postcode={50019}, country={Spain}}
\affiliation[2]{organization={Department of Mathematics, San Jos\' e State University},  city={San Jose}, postcode={95192}, country={U.S.A.}}
\affiliation[3]{organization={Departamento de F\'isica y Matem\'aticas, Universidad de Alcal\'a}, addressline={Alcal\'a de Henares}, city={Madrid}, postcode={28871}, country={Spain}}
\affiliation[4]{organization={Departamento de F\'isica y Matem\'aticas, Universidad de Alcal\'a}, addressline={Alcal\'a de Henares}, city={Madrid}, postcode={28871}, country={Spain}}
\affiliation[5]{organization={Departamento de Matem\' atica Aplicada, Universidad de Zaragoza}, addressline={Edificio de Matem\'aticas}, city={Zaragoza}, postcode={50019}, country={Spain}}
\affiliation[6]{organization={Department of Mathematics, University of California},  city={Berkeley}, postcode={94720}, country={U.S.A.}}
\affiliation[7]{organization={Sofia High School of Mathematics}, addressline={61 Iskar St.}, city={Sofia}, postcode={1000}, country={Bulgaria}}

\ignore{

\author{Jorge Delgado \and Plamen Koev \and Ana Marco \and Jos\'e-Javier Mart\'inez  \and Juan Manuel Pe\~na \and Per-Olof Persson \and Steven Spasov}
\institute{
Departamento de Matem\' atica Aplicada, Universidad de Zaragoza, Edificio Torres Quevedo, Zaragoza 50019, Spain
\\
\email{jorgedel@unizar.es}
\\
Department of Mathematics, San Jose State University, 
San Jose, CA 95192, U.S.A.
\\
\email{plamen.koev@sjsu.edu}
\and
\\
Departamento de Fisica y Matem\'aticas, Universidad de Alcal\'a, Alcal\'a de Henares, Madrid 28871, Spain
\\
\email{ana.marco@uah.es}
\and
\\
Departamento de Fisica y Matem\'aticas, Universidad de Alcal\'a, Alcal\'a de Henares, Madrid 28871, Spain
\\
\email{jjaviermartinez@uah.es}
\and
\\
Departamento de Matem\'atica Aplicada, Universidad de Zaragoza, Edificio de Matem\'aticas, Zaragoza 50009, Spain
\\
\email{jmpena@posta.unizar.es}
\and
\\
Department of Mathematics, University of California -- Berkeley, 
Berkeley, CA 94720, U.S.A.
\\
\email{persson@berkeley.edu}
\and
\\
Sofia High School of Mathematics, 61 Iskar St, Sofia 1000, Bulgaria
\\
\email{stevenspasov@gmail.com}
}

} 




\begin{abstract} 
We present a method to derive new explicit expressions for bidiagonal decompositions of Vandermonde and related matrices such as the ($q$-, $h$-) Bern\-stein-Vandermonde ones, among others. These results generalize the existing expressions for nonsingular matrices to matrices of arbitrary rank. \tr{For totally nonnegative matrices of the above classes, the new decompositions can be computed efficiently and to high relative accuracy componentwise in floating point arithmetic. In turn, matrix computations (e.g., eigenvalue computation) can also be performed efficiently and to high relative accuracy.}
\end{abstract}

\end{frontmatter}


\section{Introduction}
A matrix is totally nonnegative (TN) if all of its minors are nonnegative \cite{ando,fallatjohnsontp,karlin}. The bidiagonal decompositions of the TN matrices have become an important tool in the study of these matrices \cite{fallat01,gascapena92} and for performing matrix computations with them accurately and efficiently \cite{koev,koevTP2,koevSTN}. The bidiagonal decompositions of many classical TN matrices such as Vandermonde and many related matrices are well known, but contain singularities when some of the nodes coincide. For example, a $3\times 3$ Vandermonde matrix 
with nodes $x,y,z$ is decomposed as \cite{koev}:
\begin{align}
\bmat{ccc} 
1 & x & x^2 \\
1 & y & y^2 \\
1 & z & z^2 
\emat
&=
\bmat{ccc}
1 \\
& 1\\
& 1 & 1
\emat
\bmat{ccc}
1 \\
1 & 1\\
& \frac{z-y}{y-x} & 1
\emat
\bmat{ccc}
1 \\
& y-x\\
&  & (z-x)(z-y)
\emat
\nonumber
\\
& \phantom{=}\mbox{ }\;\times
\bmat{ccc}
1 & x\\
& 1 & y\\
&  & 1
\emat
\bmat{ccc}
1   \\
& 1& x\\
& & 1 
\emat.
\label{vand1}
\end{align}
This decomposition is undefined when $x=y$, which is unfortunate, since the Vandermonde matrix is very well defined for any values of the nodes.

By relaxing the requirement that the bidiagonal factors have ones on the main diagonal, we show how to rearrange the factors in \eqref{vand1}, so that the new decomposition contains no singularities and is valid for any values of the nodes and is thus valid for a Vandermonde matrix of arbitrary rank.
For example, the $3\times 3$ Vandermonde from \eqref{vand1} can be decomposed as 
\begin{align}
\bmat{ccc} 
1 & x & x^2 \\
1 & y & y^2 \\
1 & z & z^2 
\emat
&=
\bmat{ccc}
1 \\
& 1\\
& 1 & z-y
\emat
\bmat{ccc}
1 \\
1 & y-x \\
& 1 & z-x
\emat
\bmat{ccc}
1 \\
& 1 \\
&  & 1
\emat
\nonumber 
\\
& \phantom{=}\mbox{ }\;\times
\bmat{ccc}
1 & x\\
& 1 & y\\
&  & 1
\emat
\bmat{ccc}
1   \\
& 1& x\\
& & 1 
\emat,
\label{Vand3-2}
\end{align}
which is valid for any values of the nodes $x,y,z$.

Many classes of other TN matrices share the exact same singularities in their bidiagonal decompositions, e.g., the Bernstein-Vandermonde, their $q$-, $h$-, and rational generalizations, Lupa\c s,  Said-Ball matrices, etc.~\cite{koev,DP13,DP15qBV,delgadopena17,MM07,marcomartinez10,MM13,MMV19,marcomartinez16}. We call these matrices Vandermonde-type below (see section \ref{sec_vand}).

The method that allowed us to obtain the decomposition \eqref{Vand3-2} from \eqref{vand1} applies to all Vandermonde-type matrices (section \ref{sec_method}) and is the main contribution of this paper.
\tr{The starting point for our method is the existing bidiagonal decompositions of the Vandermonde-type matrices, which are valid only when these matrices are both TN and nonsingular. Our method takes these decompositions as a starting point and produces new bidiagonal decompositions valid for matrices of arbitrary rank, regardless of whether they are TN or not -- see Corollary \ref{Cor1}.}

While the computational complexity of the transformed bidiagonal decomposition of an $n\times n$ matrix remains $O(n^2)$, the new expressions are simpler. 

In terms of accuracy, just like their non-singular counterparts, the new decompositions remain insusceptible to subtractive cancellation, and thus all of the entries of these decompositions can be computed to high relative accuracy when the matrix is TN. By ``high relative accuracy'' we mean that for each entry its sign and most of its leading significant digits are computed correctly (see section \ref{sec_accuracy}).  

The Vandermonde-type matrices have many applications that are well
referenced in the papers we cite above that deal with the nonsingular case.
For example, the Lupa\c s matrices (section \ref{sec_Lupas}) have direct applications in CAGD \cite{delgadopena17, hanqui2014}. \tr{For TN Vandermonde-type matrices, the new results in this paper allow for matrix computations with them to be performed to high relative accuracy very efficiently (in $O(n^3)$ time) using the methods of \cite{koevSTN} now also when these matrices are singular -- see section \ref{sec_numexp} for a numerical example.}

The efficiency and high relative accuracy is particularly relevant, for example, in eigenvalue computations since the corresponding matrices are unsymmetric. The error bounds for the eigenvalues computed by 
the conventional algorithms (such as the ones in LAPACK \cite{lapackmanual3}) \cite{DemmelMA221} imply that none of the eigenvalues are guaranteed to be accurate, although the largest ones typically are -- see the example in section~\ref{sec_numexp}. In contrast, the results of this paper now allow for all eigenvalues to be efficiently computed to high relative accuracy and, in particular, the zero eigenvalues are computed exactly.

The paper is organized as follows. In section \ref{sec_bidiag} we review the bidiagonal decompositions of nonsingular TN matrices. In section \ref{sec_vand} we present the formal definition of Vandermonde-type matrix and, in section \ref{sec_sbdvand}, we show how to remove the singularities in its bidiagonal decomposition. We demonstrate how our method applies to the particular class of $q$-Bernstein-Vandermonde matrices in section \ref{sec_method}. Our method is also directly applicable to derivative matrices, such as generalized Vandermonde matrices, Laguerre matrices, etc., which are submatrices or products of Vandermonde matrices and other nonsingular TN matrices (section \ref{sec_derivative}). In section \ref{sec_nn}, we present a method to make the bottom right-hand corner entry of all bidiagonal factors in the corresponding bidiagonal decompositions equal to 1 -- this is a requirement for the methods of \cite{koevSTN} to work. We discuss accuracy issues in section \ref{sec_accuracy} and present numerical experiments in section \ref{sec_numexp}.
The explicit formulas for the decompositions of several Vandermonde-type matrices are in the Appendix.

\section{Bidiagonal decompositions of TN matrices}
\label{sec_bidiag}
Our focus in this paper is on the class of TN matrices, but the formulas and methods we present are valid without the requirement of total nonnegativity \tr{(Corollary \ref{Cor1} below)}. The bidiagonal decompositions of the TN matrices serve as a major tool in their study and computations with them. We review those here \cite{fallatjohnsontp,fallat01}.

A nonsingular $n\times n$ TN matrix $A$ can be uniquely factored as
\begin{equation}
A=L^{(1)} L^{(2)} \cdots L^{(n-1)} D U^{(n-1)} U^{(n-2)} \cdots U^{(1)},
\label{bidiag}
\end{equation}
where $L^{(i)}$ are $n\times n$ nonnegative and unit lower bidiagonal, $D$ is $n\times n$ nonnegative and diagonal, and $U^{(i)}$ are $n\times n$ nonnegative and unit upper bidiagonal. For the nontrivial entries of the factors $L^{(k)}$ and $U^{(k)}$, respectively, we have $L_{i+1,i}^{(k)}=U_{i,i+1}^{(k)}=0$ for $i< n-k$.

We call the above decomposition \eqref{bidiag}, the {\em ordinary bidiagonal decomposition} of $A$ to distinguish it from what we will define below as the singularity-free bidiagonal decomposition of a TN matrix.

The decomposition \eqref{bidiag} occurs naturally in the process of complete Neville elimination when adjacent rows and columns are used for elimination \cite{fallat01,gascapena92}. 

There are exactly $n^2$ nontrivial entries which parameterize the ordinary bidiagonal decomposition \eqref{bidiag}. Following  \cite[sec.~4]{koev}, those nontrivial entries can be conveniently stored in an $n\times n$ matrix $M$, denoted $M={\mathcal BD}(A)$, where:
\begin{align}
m_{ij} &= L_{i,i-1}^{(n-i+j)}, \quad i>j,
\nonumber
\\
m_{ij} &= U_{j-1,j}^{(n-j+i)}, \quad i<j,
 \label{mij2}
\\
m_{ii}&=D_{ii}.
\nonumber
\end{align}
Namely, for $k=1,2,\ldots, n-1$,
$$
L^{(k)}=\bmat{ccccccc}
1\\
 & \ddots\\
& m_{n-k+1,1} & 1 \\
&& m_{n-k+2,2} & 1 \\
&&& \ddots & \ddots \\
&&&&m_{nk} & 1
\emat,
\nonumber
$$
$$
U^{(k)}=\bmat{ccccccc}
1\\
 & \ddots\\
&& 1 & m_{1,n-k+1}  \\
&&& 1 & m_{2,n-k+2}  \\
&&&& \ddots & \ddots \\
&&&&& 1 & m_{nk}\\
&&&&&&1
\emat,
$$
and 
$$
D=\mbox{diag}(m_{11}, m_{22}, \ldots, m_{nn}).
$$

For $i\ne j$, the $m_{ij}$ are the multipliers of the complete Neville elimination with which the $(i,j)$ entry of $A$ is eliminated and $m_{ii}$ are the diagonal entries of $D$. 

For example, if $A$ is the $3\times 3$ Vandermonde matrix from \eqref{vand1}, then the matrix $M={\mathcal BD}(A)$ is:
\begin{equation}
M=
\bmat{ccc}
1 & x & x\\
1& y-x& y\\
1&  \frac{z-y}{y-x}  & (z-x)(z-y)
\emat.
\label{MBD}
\end{equation}

In this paper, we assume that explicit expressions for the entries of the ordinary bidiagonal decomposition \eqref{bidiag} are given. For the matrices we consider in this paper, those expressions contain singularities when some of the nodes coincide and we work to remove those singularities.

To that end, we drop the requirement that the matrices $L^{(i)}$ and $U^{(i)}$ have ones on the main diagonal. We define a new bidiagonal decomposition, which we call {\em singularity-free bidiagonal decomposition} and denote it by $\mathcal{SBD}(A)$, as
\begin{equation}
A=L_1L_2\cdots L_{n-1} D U_{n-1} U_{n-2} \cdots U_1.
\label{SBD}
\end{equation}
The matrix $D$ is nonnegative and diagonal. The factors $L_i$ and $U_i$ are nonnegative lower and upper bidiagonal, and have the same nonzero patterns as $L^{(i)}$ and $U^{(i)}$, respectively, $i=1,2,\ldots,n-1$.   Namely, $(L_k)_{i+1,i} = (U_k)_{i,i+1}= 0$ for $i<n-k$ (see theorem 2.1 in \cite{koev}).

 Following \cite{koevSTN}, the nontrivial entries of $\mathcal{SBD}(A)$ are stored in two matrices: $B$, which is $n\times n$ and $C$, which is $(n+1)\times (n+1)$:
$$
[B,C]=\mathcal{SBD}(A).
$$
As with $\mathcal{BD}(A)$, the matrix $B$ stores the nontrivial offdiagonal entries of $L_i$ and $U_i$ as well as the diagonal entries of $D$, exactly as in \eqref{mij2}:
\begin{align*}
b_{ij} &= (L_{n-i+j})_{i,i-1}, i>j,
\nonumber
\\
b_{ij} &= (U_{n-j+i})_{j-1,j}, i<j,
\\
b_{ii}&=D_{ii}.
\nonumber
\end{align*}
The matrix $C$ stores the diagonal entries of $L_i$ and $U_i$ as
$$
c_{ij} = \left\{ \begin{array}{ll}
(L_{n-i+j})_{i-1,i-1}, & i>j,\\
(U_{n-j+i})_{j-1,j-1}, & i<j.
\end{array}
\right.
$$
In this arrangement, $c_{ij},\, i>j$, is the diagonal entry in $L_{n-i+j}$ immediately above $b_{ij}$ and similarly for $i<j$ and $U_{n-j+i}$. The entries $c_{ii}, i=1,2,\ldots,n+1$ as well as the entries $c_{1,n+1}$ and $c_{n+1,1}$ are unused. This is the same construction as the one given in formula (9) in \cite{koevSTN}, except that we now allow for the $(n,n)$ entry in $L_i$ and $U_i$ to be any nonnegative number and not necessarily equal 1. This is the reason we need an $(n+1)\times (n+1)$ matrix to host the entries $c_{ij}$: the nontrivial diagonal entries of $L_1, L_2,\ldots, L_{n-1}$ are of lengths $2, 3, \ldots, n$, and similarly for the $U_i$'s.

As we explain in section \ref{sec_nn}, we can always make the $(n,n)$ entry in the $L_i$'s and $U_i$'s equal to 1, so that the algorithms of \cite{koevSTN} be used, but the formulas are more elegant without that restriction, which can be imposed after in software.

Thus, in this paper we start with the known ordinary bidiagonal decomposition $M=\mathcal{BD}(A)$ of a nonsingular matrix $A$, \eqref{bidiag}, and produce the singularity-free bidiagonal decomposition $[B,C]=\mathcal{SBD}(A)$, \eqref{SBD}, which is defined for a matrix $A$ of arbitrary rank.

The new singularity-free bidiagonal decomposition \eqref{SBD} is not unique, but this is inconsequential -- for the purposes of accurate computations, any accurate decomposition of a TN matrix as a product of nonnegative bidiagonals is an equally good input  \cite{koev}.

\section{Vandermonde-type matrices}
\label{sec_vand}

The ordinary Vandermonde matrices are the starting point of our investigation as the approach extends analogously to all other classes of matrices in this paper.

An $n\times n$ Vandermonde matrix $A$ with nodes $x_1,x_2,\ldots,x_n$ is defined as 
\begin{equation}
A=\left[
x_i^{j-1}\right]_{i,j=1}^n.
\label{def_vand}
\end{equation}
When the nodes are distinct, it has an ordinary bidiagonal decomposition~\eqref{bidiag},
$V={\cal BD}(A)$, such that \cite{koev}
\begin{equation}
\arraycolsep=1.4pt\def\arraystretch{2.2}
\begin{array}{lcll}
v_{ii}&= &\displaystyle\prod_{k=1}^{i-1}(x_i-x_k), & \quad  1\le i\le n,\\
v_{ij}  &= & \displaystyle\prod_{k=i-j}^{i-2}\frac{x_{i}-x_{k+1}}{x_{i-1}-x_k}, & \quad 1\le j< i\le n,\\
v_{ij} &= &x_i, & \quad 1\le i< j \le n.
\end{array}
\label{Vandermonde_bidiag}
\end{equation}

In the following definition, we take a more general stance, where the lower bidiagonal factors and the diagonal have additional factors, which contain no singularities.

\begin{definition}
\tr{
An $n\times n$ matrix $A$ is said to be of {\em Vandermonde-type} with nodes $x_1,x_2,\ldots,x_n$ on an (open or closed) interval $D$, if it satisfies all of the conditions below:
\begin{enumerate}
\item $a_{ij} = f_j(x_i)$, where $f_j(x)$ is a rational function of $x$ for $i,j=1,2,\ldots,n$;
\item $A$ is TN when $x_1<x_2<\cdots<x_n$ and all $x_i\in D, \; i=1,2,\ldots,n$;
\item $A$ has an ordinary bidiagonal decomposition, $M={\cal BD}(A)$, such that
\begin{equation}
m_{ij}= s_{ij}v_{ij},
\label{Vandermonde_bidiag2}
\end{equation}
for $1\le j\le i\le n$, where $v_{ij}$ are defined as in \eqref{Vandermonde_bidiag}, and the entries $s_{ij}$ for $1\le j\le i\le n$ and the entries $m_{ji}$ for $1\le j<i\le n$ are rational functions of $x_1, x_2,\ldots, x_n$, with no singularities when $x_i\in D, \, i=1,2,\ldots,n$. 
\end{enumerate}
A Vandermonde-type matrix will thus be fully specified by the entries $s_{ij}$ for $1\le j\le i\le n$ and the entries $m_{ji}$ for $1\le j<i\le n$.
}
\label{def_vtype}
\end{definition}

From the above definition, the ordinary Vandermonde matrix \eqref{def_vand} is a Van\-der\-monde-type matrix \tr{on $(0,\infty)$} with $s_{ij} = 1$ for $1\le j\le i\le n$ and $m_{ji}=v_{ji} = x_j$ for $1\le j<i\le n$.

In the literature, the Vandermonde-type matrices sometimes have $n+1$ nodes, with the indexing starting at 0, e.g., $x_0,x_1,\ldots,x_n$.  For consistency, in this paper, all Van\-der\-monde-type matrices are $n\times n$ and their nodes are $x_1,x_2,\ldots,x_n$.

\section{Singularity-free bidiagonal decompositions of Vandermonde-type matrices}
\label{sec_sbdvand}

In this section, we show how to remove the singularities in the ordinary bidiagonal decomposition \eqref{bidiag} for Vandemonde-type matrices and obtain a sin\-gu\-la\-ri\-ty-free bidiagonal decomposition \eqref{SBD}.

If $M={\mathcal BD}(A)$ is the ordinary bidiagonal decomposition \eqref{bidiag} of a Van\-der\-monde-type matrix, from \eqref{Vandermonde_bidiag2} we have $m_{ii} = s_{ii} v_{ii}, \, i=1,2,\ldots,n$, and thus the diagonal factor $D$ can be factored as $D=D_n\cdot D'$,  so that
$$
A=L^{(1)} L^{(2)} \cdots L^{(n-1)} \cdot D_n \cdot D' \cdot U^{(n-1)} \cdots U^{(1)},
$$
where 
\begin{equation}
D_n=\diag(v_{11},\ldots,v_{nn})
\label{Dn}
\end{equation}
and $D' = \diag(s_{11},\ldots,s_{nn})$.

Since there are no singularities in the product $D'\cdot U^{(n-1)} \cdots U^{(1)}$, it suffices to remove the singularities in the remaining factors, i.e., in the product 
\begin{equation}
L^{(1)} L^{(2)} \cdots L^{(n-1)}  D_n.
\label{product}
\end{equation}

\begin{theorem}
Let $L^{(i)}, \, i=1,2,\ldots, n-1$ be the lower bidiagonal factors in the ordinary bidiagonal decomposition \eqref{bidiag} of a Vandermonde-type matrix, whose nontrivial entries are defined as in \eqref{Vandermonde_bidiag2} and $D_n$ be defined as in \eqref{Dn}.
Then
$$
L_1 L_2 \cdots L_{n-1} = L^{(1)} L^{(2)} \cdots L^{(n-1)} D_n, 
$$
where
\begin{equation}
L_i\! =\! \bmat{ccccccc}
1 \\
& \ddots \\
&& \!\!\!\!\!\!\!\! 1\\
&&\!\!\!\!\!\!\!\!s_{n-i+1,1} & x_{n-i+1} - x_{n-i} \\
&&&s_{n-i+2,2} & x_{n-i+2} - x_{n-i} \\
&&&&\!\!\!\!\!\!\! \ddots & \ddots \\
&&&&&\!\!\!\!\!\!\!\!\!\!\!\!\!\!\!\!  s_{ni} & \!\!\!\!\!\!\!\! x_n-x_{n-i}
\emat \! .
\label{L_i}
\end{equation}
\end{theorem}

\begin{proof}
It suffices to observe just the initial step, i.e., that
\begin{equation}
 D_{n-1} \cdot L_{n-1}=L^{(n-1)} \cdot D_n,
\label{first_step}
\end{equation} 
where
$ D_{n-1}$ is a diagonal matrix, such that $(D_{n-1})_{ii}=1$ for $i=1,2,$ and
$$
(D_{n-1})_{ii} = \prod_{k=2}^{i-1} (x_i-x_k), \quad i=3,4,\ldots,n.
$$
In other words, the $(n-1)\times (n-1)$ bottom right principal submatrix of $D_{n-1}$ is the diagonal factor in the ordinary bidiagonal decomposition of an $(n-1)\times (n-1)$ ordinary Vandermonde matrix with nodes $x_2,x_3,\ldots,x_n$.

Since both $D_n$ and $D_{n-1}$ are diagonal and both $L^{(n-1)}$ and $L_{n-1}$ are lower bidiagonal, we have bidiagonals on each side of \eqref{first_step}. To establish that those are equal, it suffices to show that the corresponding diagonal and offdiagonal entries are the same. Since $L^{(n-1)}$ is unit lower bidiagonal, the diagonal entry in the product on the right, $L^{(n-1)}D_n$ is $(D_n)_{ii}$. Thus, for $i>1$,
$$
(D_n)_{ii} = \prod_{k=1}^{i-1} (x_i-x_k) = 
\left(\prod_{k=2}^{i-1} (x_i-x_k)\right) \cdot (x_i-x_1) = (D_{n-1})_{ii} (L_{n-1})_{ii},
$$
since $(L_{n-1})_{ii} =  x_i-x_1$ by \eqref{L_i}. Also, $(D_n)_{11}=1$, so the diagonal entries on both sides of \eqref{first_step} are equal.

Since $(L_{n-1})_{i+1,i}=s_{i+1,i}$, the offdiagonal entry in position $(i+1,i)$ on the left side of \eqref{first_step} is $s_{i+1,i}(D_{n-1})_{i+1,i+1}$. Also, from \eqref{mij2},
$$
L_{i+1,i}^{(n-1)} = m_{i+1,i} = s_{i+1,i}v_{i+1,i} = s_{i+1,i}\prod_{j=1}^{i-1}\frac{x_{i+1}-x_{j+1}}{x_i-x_j}
$$
and thus
\begin{align*}
s_{i+1,i} (D_{n-1})_{i+1,i+1} 
&= s_{i+1,i} \prod_{j=2}^i (x_{i+1}-x_j) \\
&= s_{i+1,i} \prod_{j=1}^{i-1} (x_{i+1}-x_{j+1}) \\
& =  s_{i+1,i} \prod_{j=1}^{i-1}\frac{x_{i+1}-x_{j+1}}{x_i-x_j} \prod_{j=1}^{i-1} (x_i-x_j)\\
& =  s_{i+1,i} v_{i+1,i} (D_n)_{ii}\\
& = (L^{(n-1)})_{i+1,i}\, ( D_n)_{ii},
\end{align*}
which is the $(i+1,i)$ entry on the right hand side of \eqref{first_step}.

Therefore the offdiagonal entries on each side of \eqref{first_step} are also equal and \eqref{first_step} is fully established.

Completely analogously, we establish that $D_{k-1} L_{k-1} = L^{(k-1)}D_{k}$ for $k=n-1, n-2, \ldots, 2,$ and for the product \eqref{product} we have
\begin{align*}
 L^{(1)} L^{(2)} \cdots \underline{L^{(n-1)} D_n} & = 
  L^{(1)} L^{(2)} \cdots \underline{L^{(n-2)} D_{n-1}} L_{n-1} \\
  & =
  L^{(1)} L^{(2)} \cdots \underline{ L^{(n-3)} D_{n-2}} L_{n-2}  L_{n-1} \\
  & = \cdots \\
    & = \underline{L^{(1)} D_2} L_2 \cdots L_{n-1} \\
  & = \underline{D_1 L_1} L_2 \cdots L_{n-1},\quad \mbox{ and since $D_1=I$}, \\
  & =L_1 L_2 \cdots L_{n-1}.
\end{align*}
 The factors that change on each step are underlined. 

In other words, we ``walk'' the matrix $D_n$ through the product \eqref{product} right-to-left,  canceling the denominators in each $L^{(k)}$ and factoring a slightly different diagonal matrix, $D_{k-1}$, out the left until we end up with just $D_1=I$, which we drop.
\end{proof}

\tr{
The ordinary bidiagonal decompositions exist for nonsingular TN matrices, however very simple singular TN matrices such as  $\big[\genfrac{}{}{0pt}{}{0 \;\;0}{1\;\;0}\big]$ (let alone non-TN ones) don't have ordinary bidiagonal decompositions. Thus the fact that {\em all} Vandermonde-type matrices have singularity-free bidiagonal decompositions regardless of their rank or total nonnegativity is remarkable, which we prove below.
}

\begin{corollary}
\tr{
The singularity-free bidiagonal decomposition \eqref{SBD} of a Van\-der\-monde-type matrix $A$ is valid for any complex values of the nodes $x_1,x_2,\ldots,x_n$ for which the matrices on both sides of \eqref{SBD} are defined.
}
\label{Cor1}
\end{corollary}

\begin{proof}
\tr{
Since it is derived from the ordinary bidiagonal decomposition \eqref{bidiag}, the sin\-gula\-ri\-ty-free bidiagonal decomposition \eqref{SBD} is valid when $A$ is nonsingular and TN. This remains the case when all the nodes of $A$ are strictly increasing inside an open interval $(a,b)$ contained in the interval $D$ of total nonnegativity of $A$. 
}

\tr{
The $(i,j)$th entries on each side of \eqref{SBD} are rational functions of $x_i$ for all $i, j=1,2,\ldots,n$.
Indeed, on the left hand side, $a_{ij} = f_j(x_i)$, is a rational function of $x_i$ by definition \ref{def_vtype}. On the right hand side, the $(i,j)$th entry is a polynomial in the entries of the factors of \eqref{SBD}. These entries are obtained via the construction in section \ref{sec_sbdvand} from the entries of the ordinary bidiagonal decomposition \eqref{bidiag}, which are rational functions of the nodes $x_1,x_2,\ldots,x_n$ as either quotients of minors of $A$ or products of quotients of minors of $A$ \cite[Prop.~(3.1)]{koev}.
}

\tr{Since all rational functions are meromorphic on  $\mathbb{C}$, and the equality between the $(i,j)$th entries on each side of \eqref{SBD} holds on an open interval in $\mathbb{R}$ containing $x_i$:  $x_{i-1}<x_i<x_{i+1}$, for $1<i<n$ and $a<x_1<x_2$, $x_{n-1}<x_n<b$ for $i=1, i=n$, respectively, the Identity Theorem \cite[Thm.\ 3.2.6]{ablowitzfokas03}, implies that this equality holds for any complex values of the nodes where these entries are defined.
}
\end{proof}

\tr{
As a direct application to the ordinary Vandemonde matrices for example, as is evident from the formulas in section \ref{ovand}, the above Corollary immediately implies that their singularity-free bidiagonal decomposition is valid for any complex nodes .
}

\section{Our method and explicit formulas for Vandemonde-type matrices}
\label{sec_method}

Our method for producing a singularity-free bidiagonal decomposition for an $n\times n$ Vandermonde-type matrix $A$ of arbitrary rank works as follows:
\begin{enumerate}
\item We start with the ordinary bidiagonal decomposition $M=\mathcal{BD}(A)$. Since $A$ is Vandermonde-type, the entries $m_{ij}$ factor as:
$$
m_{ij} = s_{ij} v_{ij}
$$ for $1\le j\le i\le n$, where $v_{ij}$ are defined in \eqref{Vandermonde_bidiag} and $s_{ij}$ contain no singularities.

\item The singularity-free bidiagonal decomposition $[B,C]=\mathcal{SBD}(A)$, is then


\begin{alignat}{2}
b_{ij} & =  s_{ij}, && \quad  1\le j\le i\le n,
\label{bc_def1}
\\
b_{ij} & =  m_{ij}, && \quad 1\le i< j\le n,
\label{bc_def2}
\\
c_{ij} & =  x_{i-1}-x_{i-j}, && \quad 2\le j< i\le n+1,
\label{bc_def3}
\\
c_{ij} & =  1, && \quad {j>i}.
\label{bc_def4}
\end{alignat}

\end{enumerate}

In other words, to obtain the lower bidiagonal factors $L_i$ in \eqref{SBD} from the corresponding factors $L^{(i)}$ in \eqref{bidiag}, we remove the factors $v_{ij}$ from subdiagonal elements in $L^{(i)}$ and set the diagonal entries in $L^{(i)}$ as in \eqref{bc_def3}.

To obtain the diagonal factor in \eqref{SBD} we remove the factors $v_{ii}$ from the corresponding diagonal entry $m_{ii}$ in the diagonal factor in \eqref{bidiag}, $i=1,2,\ldots,n$. 

The upper bidiagonal factors $U^{(i)}$ remain intact, $U_i=U^{(i)}, \, i=1,2,\ldots,n-1,$
as per \eqref{bc_def2} and \eqref{bc_def4}.

To fully specify the entire decomposition, $[B,C]=\mathcal{SBD}(A)$, it suffices, therefore, to specify the entries $s_{ij}, i\ge j$, and $m_{ij}, i<j$, The rest of the entries in $B$ and $C$ are specified by \eqref{bc_def2} and \eqref{bc_def4}.

\subsection{Ordinary Vandermonde matrices}
\label{ovand}
The ordinary bidiagonal decomposition of an ordinary Vandermonde matrix is given in \eqref{Vandermonde_bidiag}.

Their singularity-free bidiagonal decomposition is such that 
$$
b_{ij} = 1
$$
for $i\ge j$ and 
$$
b_{ij} = x_j
$$
for $i<j$ and $c_{ij} = x_{i-1}-x_{i-j}$ for $2\le j< i\le n+1$ and $c_{ij} = 1$ otherwise.
In other words, the factors in \eqref{SBD} are 
\begin{equation}
L_i\! =\! \bmat{ccccccc}
1 \\
& \ddots \\
&& 1\\
&&1 & x_{n-i+1} - x_{n-i} \\
&&&1 & x_{n-i+2} - x_{n-i} \\
&&&&\!\!\!\!\!\!\! \ddots & \ddots \\
&&&&&\!\!\!\!\!\!\!\!\!\!\!\!\!\!\!\!  1 & \!\!\!\!\!\!\!\! x_n-x_{n-i}
\emat,
\nonumber
\end{equation}
$D=I$ and 
$$
U_i = \bmat{ccccccc}
1 \\
& \ddots \\
&& 1 & x_{n-i}\\
&&&\ddots & \ddots \\
&&&&1& x_{n-1}\\
&&&&&1\\
\emat.
$$

The singularity-free bidiagonal decomposition of a an ordinary Vandermonde matrix can be computed to high relative accuracy in $O(n^2)$ time using the routine \verb+STNBDVandermonde+ in our package STNTool \cite{koevwebpage}.

For example, if $A$ is the $3\times 3$ Vandermonde matrix in our example \eqref{Vand3-2} from the Introduction, then its singularity-free bidiagonal decomposition $[B,C]={\mathcal SBD}(A)$ is such that 
 \begin{equation}
B=
\bmat{ccc}
1 & x & x\\
1& 1& y\\
1&  1 & 1
\emat
\nonumber
\quad\mbox{and}\quad
C=
\bmat{cccc}
1 & 1 & 1 & 1\\
1& 1& 1 & 1\\
1&  y-x & 1 &1\\
1&  z-y & z-x &1
\emat.
\nonumber
\end{equation}

The same technique applies to all Vandemonde-type matrices. We demonstrate the derivation of one non-trivial case, the $q$-Bernstein Vandemonde matrices. In the Appendix (section \ref{sec_appendix}), we list the explicit formulas for several other classes of Vandermonde-type matrices. 

\subsection{$q$-Bernstein-Vandemonde matrices}
\label{sec_qBV}
The $q$-Bernstein-Vandermonde matrices are expressed in terms of $q$-integers and $q$-binomial coefficients: given $q > 0$ and any nonnegative integer $r$, a $q$-integer $[r]$ is defined as
\begin{equation}
[r]=\left\{
\begin{array}{ll}
(1-q^r)/(1-q), & q\ne 1,\\
r, & q = 1.
\end{array}
\right.
\end{equation}
For a nonnegative integer $r$, a $q$-factorial is defined as
\begin{equation}
[r]!=\left\{
\begin{array}{ll}
[r][r-1]\cdots[1], & r\ne 1,\\
1, & r = 0.
\end{array}
\right.
\nonumber
\end{equation}
The $q$-binomial coefficient is defined as
$$
\left[\genfrac{}{}{0pt}{0}{n}{r}\right] = \frac{[n][n-1]\cdots [n-r+1]}{[r]!}
$$
for integers $n \ge r \ge  0$ and as zero otherwise. 

The $q$-Bernstein-Vandermonde matrices generalize the Bernstein-Van\-der\-monde matrices \cite{MM07,MM13}, which are the $q=1$ case. The nonsingular case was developed in \cite{DP15qBV}.

The $q$-Bernstein polynomials of degree $n$ for $0 < q \le 1$ are defined in \cite{phillips97} as
\begin{equation}
b_{i,q}^n(x)= \left[\genfrac{}{}{0pt}{0}{n}{i}\right] x^i \prod_{s=0}^{n-i-1}(1-q^sx), \quad x\in[0,1],\quad i=0,1,\ldots,n.
\label{qBP}
\end{equation}

An $n\times n$ $q$-Bernstein-Vandermonde matrix $A$ with nodes $x_1,x_2,\ldots, x_{n}$
 is defined as
\begin{equation}
A=\left[b^{n-1}_{j-1,q}(x_i)\right]_{i,j=1}^{n}
\label{qBVmatrix}
\end{equation}
and is TN when  
\begin{equation}
0\le x_1\le x_2\le \cdots\le x_n<1.
\label{ordered}
\end{equation}
When the nodes $x_i, \; i=1,2,\ldots,n$, are also distinct, it has an ordinary bidiagonal decomposition $M=\mathcal{BD}(A)$ with parameters \cite{DP15qBV}
\begin{align}
 m_{ij} & = \frac {1-q^{n-j}x_{i-j} }{ 1-q^{n-j} x_{i-1} } 
 \prod _{s=0}^{n-1-j}
\frac{1-q^sx_i} { 1-q^s x_{i-1}}  \underbrace{
\prod_{k=i-j}^{i-2}\frac{x_{i}-x_{k+1}}{x_{i-1}-x_k}}_{v_{ij}},  
\nonumber
\\
 m_{ji} & = \frac{[n-i+1]}{[i-1]}\cdot \frac{x_j}{1-q^{n-i}x_j} \prod_{k=1}^{j-1} \frac{1-q^{n-i+1}x_k}{ 1-q^{n-i}x_k},
\label{mij_qBV}
\end{align}
for $1\le j< i\le n$ and
\begin{align}
m_{ii} & = \left[\genfrac{}{}{0pt}{0}{n-1}{i-1}\right] \frac{ \prod _{s=0}^{n-1-i}
(1-q^sx_i)}{ \prod _{k=1}^{i-1} 
(1-q^{n-i}x_k)}\underbrace{\prod_{k=1}^{i-1}(x_i-x_k)}_{v_{ii}}
\nonumber
\end{align}
for $1\le i\le n$.

With the entries of the ordinary bidiagonal decomposition of the ordinary Vandermonde matrix $v_{ij}$ defined as in \eqref{Vandermonde_bidiag}, we have that $m_{ij} = s_{ij} v_{ij}$ for $1\le j\le i\le n$, where 
\begin{align}
 s_{ij} & = \frac {1-q^{n-j}x_{i-j} }{ 1-q^{n-j} x_{i-1} } 
 \prod _{s=0}^{n-1-j}
\frac{1-q^sx_i} { 1-q^s x_{i-1}}, 
\label{sij}
 \\
s_{ii} & = \left[\genfrac{}{}{0pt}{0}{n-1}{i-1}\right] \frac{ \prod _{s=0}^{n-1-i}
(1-q^sx_i)}{ \prod _{k=1}^{i-1}
(1-q^{n-i}x_k)}.
\label{sij2}
\end{align}

\tr{
The $q$-Bernstein-Vandermonde matrix is therefore a Vandermonde-type matrix on $[0,1)$.
}

As described at the beginning of this section, we obtain the singularity-free bidiagonal decomposition of $A$ by removing the factors $v_{ij}$ from $m_{ij}$ for $i\ge j$ and setting 
$
c_{ij} = x_{i-1}-x_{i-j}
$ for $2\le j< i\le n+1$.

The decomposition $[B,C]=\mathcal{SBD}(A)$ is thus fully defined as in \eqref{bc_def1}--\eqref{bc_def4} with the $s_{ij},\, i\le j$, defined in \eqref{sij} and \eqref{sij2} and $m_{ji},\, i>j,$ defined in \eqref{mij_qBV}.

The singularity-free bidiagonal decomposition of a $q$-Bernstein-Vandermonde matrix can be computed to high relative accuracy in $O(n^2)$ time using the routine \verb+STNBDqBernsteinVandermonde+ in our package STNTool \cite{koevwebpage}.

\section{Derivative Vandermonde matrices}
\label{sec_derivative}

Our results apply to other classes of matrices which are not directly Van\-der\-monde-type. These include
\begin{itemize}
\item submatrices of Vandermonde-type matrices, 
\item products of Vandermonde-type matrices and other matrices, whose ordinary bidiagonal decompositions are known and contain no singularities. 
\end{itemize}
The former include the generalized Vandermonde matrices. The latter include the Laguerre matrices, the Bessel matrices, and Wronskian matrices of a basis
of exponential polynomials. We address these separately.

The {\em generalized Vandermonde} matrices 
$$
G_\lambda(x_1,\ldots,x_n) = \left[x_i^{j-1+\lambda_{n-j+1}} \right]_{i,j=1}^n
$$
are defined for integer {\em partitions} $\lambda = (\lambda_1,\lambda_2,\ldots,\lambda_n)$, where $\lambda_i$ are integers such that $\lambda_1\ge \lambda_2\ge \cdots \ge \lambda_n\ge 0$. They are submatrices of ordinary Vandermonde matrices (obtained by removing appropriate columns). Their singularity-free bidiagonal decompositions can thus be obtained by starting with the singularity-free bidiagonal decomposition of the ordinary Vandermonde matrix, derived in section \ref{ovand}, and removing the appropriate columns employing the methods of \cite{koevSTN}.

The singularity-free bidiagonal decomposition of a generalized Vandermonde matrix can be computed  to high relative accuracy in $O(\lambda_1n^2)$ time using the \verb+STNBDGeneralizedVandermonde+ routine in our package STNTool \cite{koevwebpage}. 

Several classes of TN matrices are products of a TN ordinary Vandermonde matrix (call it $V$) and a second TN matrix (call it $A$) whose ordinary bidiagonal decomposition is known and contains no singularities. These include:
\begin{itemize}
\item {\em Laguerre} \cite{DOP19Laguerre}, 
\item {\em Bessel} matrices \cite{DOP19Bessel}, and 
\item {\em Wronskian matrices of a basis
of exponential polynomials} \cite{MPR19-W}.
\end{itemize}
The singularity-free bidiagonal decomposition of the product $VA$ (or $AV$) can be obtained using the method for computing the singularity-free bidiagonal decomposition of a product of TN matrices from \cite[sec.\ 6]{koevSTN} (implemented in the routine \verb+STNProduct+ in our package STNTool \cite{koevwebpage}) with the singularity-free bidiagonal decomposition of $V$  from section \ref{sec_vand} and the ordinary bidiagonal decomposition of~$A$. 

\section{The $(n,n)$ entry of the bidiagonal factors must equal 1 for computations}
\label{sec_nn}

The algorithms of \cite{koevSTN} require that the $(n,n)$ entry of every bidiagonal factor in the singularity-free bidiagonal decomposition of a TN matrix be equal to 1. These algorithms rely heavily on this fact and it turns out that the assumption that this is always the case can be made without any loss of generality.

To that end, in this section, we present a method to ``fix'' any singularity-free bidiagonal decomposition \eqref{SBD} so that the bottom right-hand corner entries of all bidiagonal factors, $L_1, L_2, \ldots, L_{n-1}$ and $U_1, U_2, \ldots, U_{n-1}$ equal to 1. 

We do so by ``moving'' the non-unit $(n,n)$ entries of each lower bidiagonal factor into $D$ and adjusting the $(n,n-1)$ entries accordingly:

For any bidiagonal matrix $L_i$ and diagonal matrix $D_i=\diag(1,1,\ldots,1,x_i),$  $ i=1,2,\ldots,n-1,$ we can make the bottom right hand corner entry of the product $D_i L_i$ equal to one by factoring a new diagonal factor $D_{i+1}$ out of the right. Namely,
\begin{equation}
D_i L_i = \overline L_i D_{i+1},
\label{DLLD}
\end{equation}
where $\overline L_i$ equals $L_i$, except for $(\overline L_i)_{nn}=1,\, (\overline L_i)_{n,n-1} = L_{n,n-1} x_i$. Then $D_{i+1}=\diag(1,1,\ldots, x_i L_{nn})$.

This allows us, by starting with $D_1 = I$, to move the bottom right hand corner entries of the lower bidiagonal factors of $\mathcal{SBD}(A)$ into the diagonal factor,~$D$:
\begin{align}
A&=L_1 L_2 \cdots L_{n-1} \cdot D \cdot U_{n-1} U_{n-2} \cdots U_{1}
\label{SBD2}
\\
&=\underline{D_1 L_1} L_2 \cdots L_{n-1} \cdot D \cdot U_{n-1} U_{n-2} \cdots U_{1}, \quad\mbox{since }D_1=I,
\nonumber
\\
&=\overline L_1 \underline{ D_2 L_2} \cdots L_{n-1} \cdot D \cdot U_{n-1} U_{n-2} \cdots U_{1},\quad\mbox{using }D_1 L_1=\overline L_1 D_2\mbox{ as in \eqref{DLLD},}
\nonumber
\\
&=\ldots
\nonumber
\\
&=\overline L_1 \overline L_2 \cdots  \underline{D_{n-1}L_{n-1}} \cdot D \cdot U_{n-1} U_{n-2} \cdots U_{1}
\nonumber
\\
&=\overline L_1 \overline L_2 \cdots \overline L_{n-1} \cdot (D_n D)\cdot U_{n-1} U_{n-2} \cdots U_{1}
\nonumber
\end{align}
(the factors being transformed on each step using \eqref{DLLD} are underlined).

The procedure is analogous for the upper bidiagonal factors.

Thus, if \eqref{SBD2} is a singularity-free bidiagonal decomposition of the TN matrix $A$, then 
$$
A=\overline L_1 \overline L_2 \cdots \overline L_{n-1} \cdot \overline D  \cdot \overline U_{n-1} \overline U_{n-2} \cdots \overline U_{1}
$$
is a singularity-free idiagonal decomposition of $A$ with all $(n,n)$ entries of $\overline L_i$ and $\overline U_i$, $i=1,2, \ldots,n$, equal to~1.

The matrices $\overline L_i, \overline U_i, i=1,2,\ldots, n-1$, and $\overline D$ differ from the matrices $L_i, U_i$ and $D$, respectively, only in 
$$(\overline L_i)_{n,n}=(\overline U_i)_{n,n}=1$$ 
and 
\begin{align*}
(\overline  L_i)_{n,n-1} &= (L_i)_{n,n-1} \prod_{k=1}^{i-1} (L_k)_{n,n},\\
(\overline U_i)_{n,n-1} &= (U_i)_{n,n-1} \prod_{k=1}^{i-1} (U_k)_{n,n},\\
\overline D_{nn} & = D_{nn} \prod_{k=1}^{n-1} (L_k)_{n,n} \prod_{k=1}^{n-1} (U_k)_{n,n}.
\end{align*}

The routine \verb+STNFixBottomRightOfBD+ in our package STNTool \cite{koevwebpage} implements the technique from this section.

For example, by implementing the technique from this section to the sin\-gu\-larity-free bidiagonal decomposition of our $3\times 3$ example \eqref{Vand3-2} we obtain a new singularity-free bidiagonal decomposition with all $(n,n)$ entries of all bidiagonal factors now equal to 1:
\begin{align}
\bmat{ccc} 
1 & x & x^2 \\
1 & y & y^2 \\
1 & z & z^2 
\emat
&=
\bmat{ccc}
1 \\
& 1\\
& 1 & 1
\emat
\bmat{ccc}
1 \\
1 & y-x \\
& 1 & 1
\emat
\bmat{ccc}
1 \\
& 1 \\
&  & (z-y)(z-x)
\emat
\nonumber 
\\
& \phantom{=}\mbox{ }\;\times
\bmat{ccc}
1 & x\\
& 1 & y\\
&  & 1
\emat
\bmat{ccc}
1   \\
& 1& x\\
& & 1 
\emat.
\nonumber
\end{align}

\section{Numerical accuracy}
\label{sec_accuracy}
In the standard ``$1+\delta$'' model of floating point arithmetic \cite{higham02}, to which the IEEE 754 double precision arithmetic \cite{IEEE754} conforms, the result of any floating point calculation is assumed to satisfy
\begin{equation}
\mbox{fl} (a \odot b) = (a\odot b)(1+\delta),
\label{1plusdelta}
\end{equation}
where $\odot \in \{+,-\times,/\},\, |\delta|\le \varepsilon$, and $\varepsilon$ is tiny and is called {\em machine precision}.

For a computed quantity, $\hat x$ to {\em have high relative accuracy}, it means that it satisfies an error bound with its true counterpart, $x$
$$
|\hat x - x|\le \theta |x|,
$$
where $\theta$ is a modest multiple of $\varepsilon$. In other words, the sign and most significant digits of $x$ must be correct. In particular, if $x=0$, it must be computed exactly.

The above model directly implies that the accuracy in numerical calculations is lost due to one phenomenon only, known as subtractive cancellation \cite{demmel98}. It occurs when a subtraction of previously rounded off quantities results in the loss of significant digits.
Multiplication, division, and addition of same-sign quantities preserve the relative accuracy. Subtraction of initial data such as the nodes $x_i$ is also fine, since initial data is assumed to be exact: \eqref{1plusdelta} tells us the result of that subtraction is computed to high relative accuracy. The only subtractions in any singularity-free bidiagonal decomposition presented in this paper are between exact initial data: either in the form  $x_i-x_j$ between the nodes $x_i$ of the Vandermonde-type matrix as in \eqref{bc_def3} or $1-x_i$ between the exact double precision floating point number $1$ and a node $x_i$ as in \eqref{1minusxi}, \eqref{1minusxi2}, and~\eqref{1minusxi3}. 

Detailed error analyses for the ordinary bidiagonal decompositions of all nonsingular TN Vandermonde-type matrices presented in this paper have already been performed in the corresponding papers (e.g., \cite{MM13} for the Bern\-stein-Vandermonde matrices, etc.). All those decompositions are computable to high relative accuracy componentwise.

The new singularity-free bidiagonal decompositions inherit the same componentwise error bounds and are thus also computable to high relative accuracy: the offdiagonal entries in the bidiagonal factors $L_i$ and $U_i$ as well as the diagonal entries of the diagonal factor $D$ in the singularity-free bidiagonal decomposition \eqref{SBD} are the same as the corresponding entries in the ordinary bidiagonal decomposition \eqref{bidiag} except for the factors $v_{ij}$ from \eqref{Vandermonde_bidiag}. The diagonal entries in the bidiagonal factors $L_i$ and $U_i$ in \eqref{SBD} are either equal to 1 or are of the form $x_i-x_j$, the latter computable with relative error bounded by the machine precision, $\varepsilon$, per \eqref{1plusdelta}.

 \section{Numerical experiments}
 \label{sec_numexp}
 
 We performed extensive numerical tests to verify the correctness of the formulas we derived in this paper as well as their inherent accuracy in numerical computations. 
 We present one illustrative example in Figure \ref{fig_all}. Using the formulas in section \ref{sec_qBV}, we computed the eigenvalues 
 of the $24\times 24$ $q$-Bernstein-Vandermonde matrix $A$ with parameter $q=0.1$ and nodes
$$
0.1,0.2,0.2,0.2,0.3,0.31,0.32,0.33,0.34,0.35,0.36,0.37,0.38,0.39,$$
$$
0.5,0.6,0.7,0.7,0.7,0.7,0.7,0.7,0.8,0.9.
$$
 
 \begin{figure}[h]
\centerline{\resizebox{3in}{2.5in}{\includegraphics*{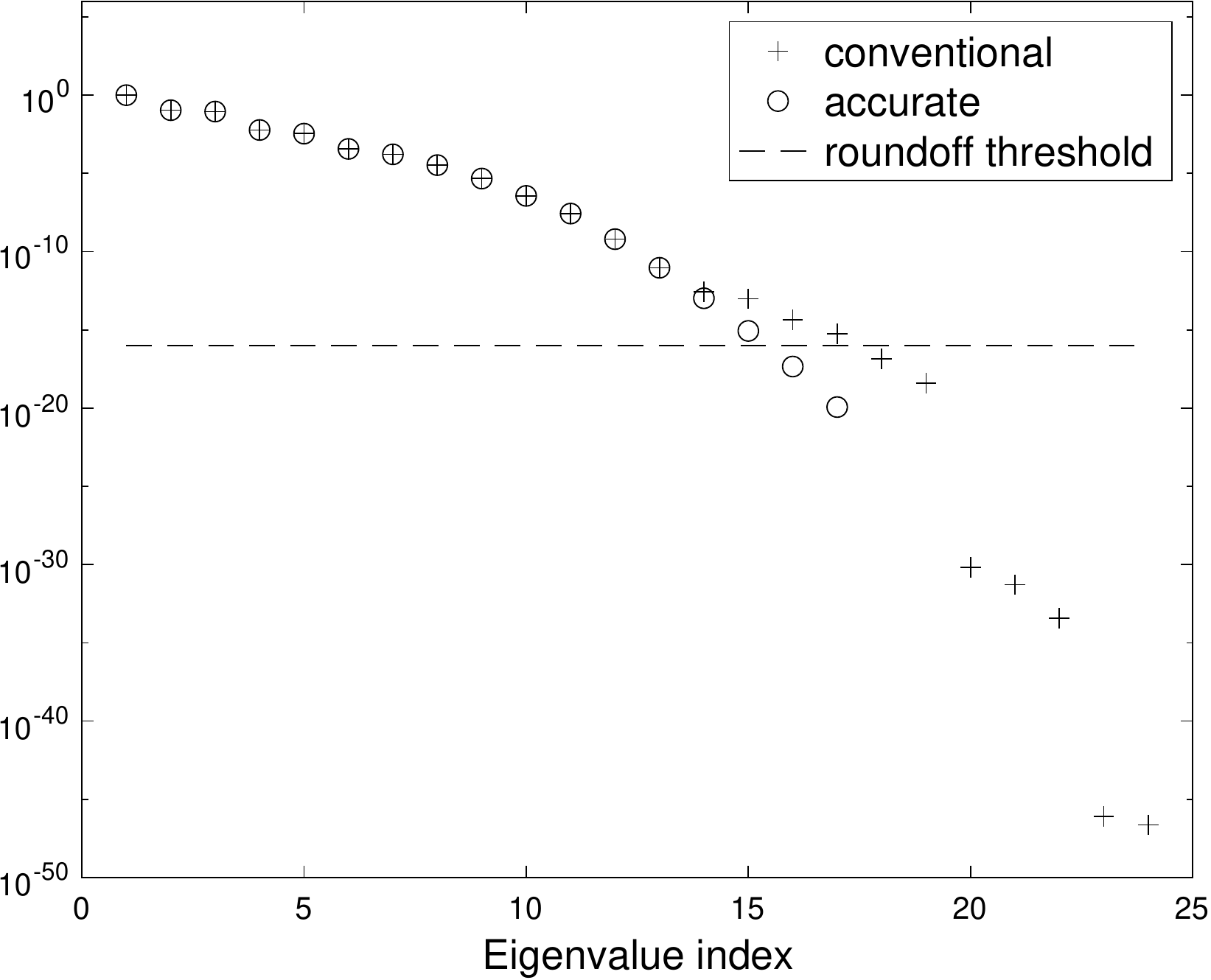}}}
\caption{The eigenvalues of a $24\times 24$ $q$-Bernstein-Vandermonde matrix as computed by the conventional eigenvalue algorithms of LAPACK as implemented by MATLAB's {\tt eig} and also by forming the accurate singularity-free bidiagonal decomposition as in section \ref{sec_method} and then using the accurate eigenvalue algorithm of \cite{koevSTN}.}
\label{fig_all}
\end{figure}
 
A $q$-Bernstein-Vandermonde matrix with distinct nodes is nonsingular \cite{DP15qBV}, thus the 17 rows corresponding to distinct nodes are linearly independent. The repeated nodes ($0.2$, three times, and $0.7$, six times) correspond to repeated rows in the matrix and thus the rank of the matrix is~17.
 
We started with the singularity-free bidiagonal decomposition of $A$, computed using the formulas in section \ref{sec_qBV}, and implemented in our software package STNTool  \cite{koevwebpage} as the routine \verb+STNBDqBernsteinVandermonde+. We computed the eigenvalues using the algorithm \verb+STNEigenValues+ \cite{koevSTN}. The matrix $A$ is never explicitly formed or needed in this part of the computation, which is not only accurate (as we see below), but also very efficient: The singularity-free bidiagonal decomposition takes $O(n^2)$ time and the eigenvalues take an additional $O(n^3)$ time \cite{koevSTN}.

For comparison, in double precision floating point arithmetic \cite{IEEE754}, we formed $A$ explicitly and computed its eigenvalues using the conventional eigenvalue algorithm of LAPACK \cite{lapackmanual3} (as implemented by \verb+eig+ in MATLAB \cite{matlab}). As expected, only the largest eigenvalues of $A$ are computed accurately by \verb+eig+. Further, since $A$ is unsymmetric and its largest eigenvalue is about 1, the complete loss of relative accuracy occurs in all eigenvalues smaller than about $10^{-12}$, a value much larger machine precision (about $10^{-16}$). The zero eigenvalues are lost to roundoff. This behavior is fully expected and justified for the general, structure-ignoring algorithms of LAPACK. It also underscores the utility of developing special algorithms for structured matrices, which deliver results to high relative accuracy with the same $O(n^3)$ efficiency.

For verification, we formed $A$, then computed its eigenvalues, in 60 decimal digit arithmetic using the software package Mathematica. All nonzero eigenvalues computed by \verb+STNEigenValues+ agreed with those computed by Mathematica to at least $14$ significant decimal digits. No amount of extra precision will allow us to reliably compute the zero eigenvalues using the conventional eigenvalue algorithms of LAPACK, thus the only reason we know the algorithm from \cite{koevSTN} computed the correct number of zero eigenvalues ($7$) is because we know that the rank of the matrix is 17.\footnote{Since the graph is log-scale, the zero eigenvalues are not depicted and the negative and complex eigenvalues returned by the conventional algorithm are displayed by their absolute values.} 

MATLAB implementations of the algorithms described in this paper are available online~\cite{koevwebpage}.

\section{Appendix}
\label{sec_appendix}

Here we present the explicit formulas for the singularity-free bidiagonal decompositions of the following Vandermonde-type matrices:
\begin{itemize}
\item $h$-Bernstein-Van\-der\-monde
\item Lupa\c s
\item rational Bern\-stein-Van\-der\-monde, and
\item Cauchy-Vandermonde matrices with one multiple pole.
\end{itemize}
This is not a complete list of Vandermonde-type matrices as this is an active area of research, but since we derived these formulas before realizing that all matrices of the Vandermonde-type follow the same pattern, we share them here and have implemented them in software \cite{koevwebpage}.

\subsection{Lupa\c{s} matrices}
\label{sec_Lupas}

The Lupa\c{s} $q$-analogues of the
 Bernstein functions of degree $n$ for $0<q \le 1$ are defined as \cite{lupas87}
 $$
 l_{i,q}^n(x)=\frac{a_{i,q}^n(x)}{w_q^n(x)}, \quad x\in [0,1],\quad i=0,1,\ldots,n, 
 $$
where
\begin{align*}
a_{i,q}^n(x) &= \left[\genfrac{}{}{0pt}{0}{n}{i}\right]  q^{\frac{i(i-1)}{2}} x^i (1-x)^{n-i}, \\
w_q^n(x) & = \sum_{i=0}^n a_{i,q}^n(x) = \prod_{i=2}^n(1-x+q^{i-1}x).
\end{align*}
For case $q = 1$ the Lupa\c{s} $q$-analogues of the Bernstein functions coincide with the Bernstein polynomials.

An $n\times n$ Lupa\c{s} matrix $A$ with nodes $x_1,x_2,\ldots,x_n$ is defined in \cite{delgadopena17} as 
$$
A=\left[l^{n-1}_{j-1,q}(x_i)\right]_{i,j=1}^n.
$$
For $0\le x_1\le x_2\le\cdots\le x_{n}<1$, the matrix $A$ is TN  \cite[Thm.\ 2.1]{delgadopena17}. It is also a Vandermonde-type matrix on \tr{$[0,1)$} with a singularity-free bidiagonal decomposition $[B,C]=\mathcal{SBD}(A)$ defined as in \eqref{bc_def1}--\eqref{bc_def4} with parameters 
\begin{align*}
s_{ij} & = \frac{(1-x_i)^{n-j}(1-x_{i-j-1})}
{(1-x_{i-1})^{n+1-j}}
\prod_{k=1}^{n-2}\frac{1-x_{i-1}+q^{k}x_{i-1}}{1-x_i+q^{k}x_i}
,  
\\ 
m_{ji} & = \frac{[n-i+1]q^{i-2}x_j}{[i-1](1-x_j)},\\
\intertext{for $1\le j< i\le n$ and}
s_{ii} &=\left[\genfrac{}{}{0pt}{0}{n-1}{i-1}\right]   \frac{q^{\frac{(i-1)(i-2)}{2}}(1-x_i)^{n-i}}{\prod_{k=2}^{n-1}(1-x_i+q^{k-1}x_i)\prod_{k=0}^{i-1}(1-x_k)}
\end{align*}
for $1\le i\le n$.

The singularity-free bidiagonal decomposition of a Lupa\c{s} matrix can be computed to high relative accuracy in $O(n^2)$ time using the routine \verb+STNBDLupas+ in our package STNTool \cite{koevwebpage}.

\subsection{$h$-Bernstein-Vandermonde matrices}
\label{sec_hBV}

These matrices are generalization of the Bernstein-Vandermonde matrices \cite{MM07,MM13}, which are the $h=0$ case. The nonsingular case was developed in \cite{MMV19}.

The $h$-Bernstein polynomials of degree $n$ for a real parameter $h \ge 0$ are defined in \cite{MMV19} as
\begin{equation}
b_{i,h}^n(x)= \left(\genfrac{}{}{0pt}{0}{n}{i}\right)\frac{\prod_{k=0}^{i-1}(x+kh)\prod_{k=0}^{n-i-1}(1-x+kh)}{\prod_{k=0}^{n-1}(1+kh)}, \; x\in[0,1],\; 0\le i\le n.
\label{hBP}
\end{equation}

An $n\times n$ $h$-Bernstein-Vandermonde matrix with nodes $x_1,x_2,\ldots, x_{n}$
 is defined as
$$
A=\left[b^{n-1}_{j-1,h}(x_i)\right]_{i,j=1}^{n}.
$$
For $0\le x_1\le x_2\le \cdots\le x_n<1$, $A$ is a TN Vandermonde-type matrix \cite{MMV19}  on \tr{$[0,1)$} and has a singularity-free bidiagonal decomposition $[B,C]=\mathcal{SBD}(A)$ defined as in \eqref{bc_def1}--\eqref{bc_def4} with parameters 
\begin{align}
s_{ij} & = \frac{ 1-x_{i-j}+(n-j)h} {1-x_{i-1}+(n-j)h} \prod_{k=0}^{n-j-1} \frac{1-x_i+kh}{1-x_{i-1}+kh},
\label{1minusxi}
\\
m_{ji} &= \frac{n-i+1}{i-1}\cdot \frac{ (x_j+(i-j-1)h)\prod_{k=1}^{j-1}(1-x_k+(n-i+1)h)}{\prod_{k=1}^j(1-x_k+(n-i)h)},
\nonumber
\\
\intertext{for $1\le j< i\le n$ and}
s_{ii} &=  \left(\genfrac{}{}{0pt}{0}{n-1}{i-1}\right)
\frac{\prod_{k=0}^{n-i-1}(1-x_i+kh)}{\prod_{k=1}^{n-i-1}(1+kh)\prod_{k=1}^{i-1}(1-x_k+(n-i)h)}
\nonumber
\end{align}
for $1\le  i\le n$.

The singularity-free bidiagonal decomposition of an $h$-Bernstein-Van\-der\-mon\-de matrix can be computed to high relative accuracy in $O(n^2)$ time using the routine \verb+STNBDhBernsteinVandermonde+ in our package STNTool \cite{koevwebpage}.

\subsection{Rational Bernstein-Vandermonde matrices} 
\label{sec_rbv}

A rational Bernstein-Vandermonde matrix with nodes $x_1, x_2,\ldots, x_n$ and positive weights $w_1,w_2,\ldots,w_n$ is defined as \cite{DP13}:
$$
R=\left[r_{j-1}^{n-1}(x_i)\right]_{i,j=1}^n,
$$
where
\begin{equation}
 r_{j-1}^{n-1}(x)=\frac{w_{j}\,b_{j-1}^{n-1}(x)}{W(x)},\quad x\in [0,1],\quad j\in\{1,2,\ldots,n\},
\label{def_rbv}
\end{equation}
and  $W(x)=\sum_{j=1}^{n} w_{j}\,b_{j-1}^{n-1}(x)$. The functions $b_{j}^{n}(x)$ are the Bernstein polynomials of degree $n$. They equal the $q$-Bernstein polynomials \eqref{qBP} for $q=1$ and the $h$-Bernstein polynomials \eqref{hBP} for $h=0$, i.e., 
\[ 
 b_{i}^{n}(x)=
\left(\genfrac{}{}{0pt}{0}{n}{i}\right)
x^i(1-x)^{n-i},\quad i\in\{0,1,\ldots,n\},\quad x\in [0,1].
\]
 For $0< x_1\le x_2\le \cdots\le x_n<1$, the matrix $R$ is TN \cite[Thm.\ 3.1]{DP13}, which is of Vandermonde-type on \tr{$(0,1)$} and
 has a singularity-free bidiagonal decomposition $[B,C]=\mathcal{SBD}(R)$ defined as in \eqref{bc_def1}--\eqref{bc_def4} with parameters 
\begin{align}
s_{ij}&=\frac{ W(x_{i-1})(1-x_i)^{n-j} (1-x_{i-j})
 }
{W(x_i)(1-x_{i-1})^{n-j+1}
 },
 \label{1minusxi2}
\\
m_{ji}&=\frac{w_i}{w_{i-1}}\cdot \frac{n-i+1}{i-1}\cdot \frac{x_j}{1-x_j},
 \label{1minusxi3}
\\
\intertext{for $1\le j< i\le n$ and}
s_{ii}&= \left(\genfrac{}{}{0pt}{0}{n-1}{i-1}\right)\frac{w_i(1-x_i)^{n-i}
}{ W(x_i)\prod_{k=1}^{i-1}(1-x_k)}
\nonumber
\end{align}
for $1\le i\le n$.

While we like the above expressions for their symmetry, an alternative sin\-gu\-la\-ri\-ty-free bidiagonal decomposition of $R$ can be obtained by recognizing, from \eqref{qBVmatrix} and \eqref{def_rbv}, that 
$R=W_1^{-1} A W_2,$ where $A$ is the $q$-Bernstein-Vandermonde matrix for $q=1$, and 
\begin{align*}
W_1&=\diag(W(x_1),W(x_2), \ldots, W(x_n))\\
W_2&=\diag(w_1,w_2,\ldots,w_n)
\end{align*}
are positive diagonal matrices. Thus if $\mathcal{SBD}(A)$ is given by \eqref{SBD}, then $R$ has a singularity-free  bidiagonal decomposition 
\begin{equation}
R=(W_1^{-1}L_1)L_2\cdots L_{n-1} D U_{n-1} U_{n-2} \cdots (U_1 W_2),
\label{RBValt}
\end{equation}
where the first and last factors in the parentheses, $W_1^{-1}L_1$ and $U_1 W_2$,, are nonnegative bidiagonal matrices with the same nonzero pattern as $L_1$ and $U_1$, respectively. Thus \eqref{RBValt} is another bidiagonal decomposition of $R$ and an example of how the nonuniqueness of the singularity-free bidiagonal decomposition \eqref{SBD} can play out.

The singularity-free bidiagonal decomposition of a rational Bernstein-Van\-der\-monde matrix can be computed to high relative accuracy in $O(n^2)$ time using the routine \verb+STNBDRationalBernsteinVandermonde+ in our package STNTool \cite{koevwebpage}.

\subsection{Cauchy--Vandermonde matrices with one multiple pole} 

A Cauchy-Vandermonde matrix with nodes $x_1,\ldots,x_n$ and one pole $d$ of multiplicity $s\ge 1$ is defined as \cite{martinezpena03}
$$
A=\left[
  \begin{array}{ccccccc}
    \frac{1}{(x_1+d)^s} & \cdots & \frac{1}{x_1+d} & 1 & x_1 & \cdots & x_1^{n-s-1} \\
    \frac{1}{(x_2+d)^s} & \cdots & \frac{1}{x_2+d} & 1 & x_2 & \cdots & x_2^{n-s-1} \\
    \vdots &  & \vdots & \vdots & \vdots &  & \vdots\\
    \frac{1}{(x_n+d)^s} & \cdots & \frac{1}{x_n+d} & 1 & x_n & \cdots & x_n^{n-s-1} \\
  \end{array}
\right]
$$
and is totally nonnegative when $0\le x_1\le x_2 \le \cdots \le x_n$ and $d>0$. It is a Vandermonde-type matrix on \tr{$[0,\infty)$ for $d>0$} and has a decomposition $[B,C]=\mathcal{SBD}(A)$ with parameters 
\begin{align*}
s_{ij} & = \left(\frac{x_{i-1}+d}{x_i+d}\right)^s,  
\\ 
m_{ji} & = \left\{
                                       \begin{array}{ll}
                                         x_j+d, & \hbox{$0<i-j\leq s$} \\
                                         x_j, & \hbox{$i-j> s$}
                                       \end{array}
                                     \right.
,\\
\intertext{for $1\le j< i\le n$ and}
s_{ii} &=\frac{1}{(x_i+d)^s}
\end{align*}
for $1\le i\le n$.

The singularity-free bidiagonal decomposition of a Cauchy-Van\-der\-monde matrix with one multiple pole can be computed to high relative accuracy in $O(n^2)$ time using the routine \verb+STNBDCauchyVandermonde1pole+ in our package STNTool \cite{koevwebpage}.

\subsection{Other Vandermonde-type matrices}

The method presented in this section can be used analogously to derive sin\-gu\-lar\-i\-ty-free bidiagonal decompositions of other Vandermonde-type matrices, for example:
\begin{itemize}
\item Said-Ball Vandermonde \cite{marcomartinez10,marcomartinez16},
\item Rational Said-Ball Vandermonde \cite{DP13},
\item Collocation and Wronskian matrices of Jacobi polynomials \cite{MPR19-WJ}.
\end{itemize}
This is likely an incomplete list as this currently appears to be an area of active research.

\section{Acknowledgements}

\tr{ 
We thank the anonymous referees for the very careful reading of our manu\-script as well as their very constructive comments and suggestions, which greatly improved the presentation and clarity of the paper and, in particular, for pointing out the significance of the result now formulated in Corollary \ref{Cor1}. We also thank Slobodan Simi\'c for his help with the proof of that corollary.
}

P.\ Koev thanks the University of California -- Berkeley for the kind hospitality during his sabbatical leave in 2019.

\bibliographystyle{elsarticle-num}
\bibliography{../SJSU/papers/biblio}

\end{document}